\documentclass[a4paper,11pt,leqno]{amsart}
\usepackage{amsmath,amscd,amsfonts,amssymb}
\usepackage{latexsym}
\usepackage[english]{babel}

\numberwithin{equation}{section}
\newtheorem{theorem}{Theorem}[section]
\newtheorem{prop}[theorem]{Proposition}
\newtheorem{cor}[theorem]{Corollary}
\newtheorem{lemma}[theorem]{Lemma}

\theoremstyle{definition}
\newtheorem{defi}{Definition}

\newtheorem{rem}[theorem]{Remark}

\newcommand{\RR}{\mathbb{R}}
\newcommand{\ZZ}{\mathbb{Z}}

\title{Higher dimensional M\"obius bands and their boundaries}
\author{Chady El Mir}
\email{chady.mir@gmail.com}

\address{Current address: \newline
Universit\'e Libanaise, Laboratoire de Math\'ematiques et Applications (LaMA-Liban), Tripoli, Liban}

\author{Jacques Lafontaine}
\email{jacques.lafontaine@univ-montp2.fr}

\address{Current address: \newline
Universit\'e de Montpellier, Institut Montpelli\'erain Alexander Grothendieck (UMR 5149), CC 0051, Place Eug\`ene Bataillon, F-34095, Montpellier Cedex 5, France}

\begin{document}
\date{}

\selectlanguage{english}
\begin{abstract}

We give a characterisation of Bieberbach manifolds which are geodesic boundaries of a compact flat
manifold, and discuss the low dimensional cases, up to dimension 4.

\end{abstract}
\medskip

\maketitle

\selectlanguage{english} \noindent {\bf Keywords}: flat manifolds, geodesic boundary, Bieberbach group.
\\
\newline
{\bf 2010 MSC}: 53C20, 53C22, 53C23.
\\
\newline

\section{Introduction}

\subsection{Presentation of our results}

A flat Riemannian manifold is a Riemannian manifold locally isometric to the Euclidean space $E^n$. If such a manifold is complete, then the exponential map is a Riemannian covering: complete flat manifolds are the quotients of the Euclidean space by a sub-group of affine isometries acting properly and freely.\\

In particular, flat compact manifolds are the quotients $E^n/\Gamma$, where $\Gamma$ is a discrete, co-compact and fixed points free group of affine isometries of $E^n$. Following \cite{charlap}, we call such groups \emph{Bieberbach groups}.\\

Bieberbach Theorem shows, more generally, that a \emph{crystallographic group}, i.e., a discrete and co-compact sub-group $\Gamma$ of affine isometries of $E^n$, is an extension  of a finite group $G$  by a lattice $\Lambda$ of $E^n$.
This lattice is the sub-group of elements of $\Gamma$ that are translations. Note that a crystallographic group is a Bieberbach group if and only
if it is torsion free (see \cite{wolf}, p.99).

\smallskip

\begin{theorem}[Bieberbach Theorem] i) Let $\Gamma$ be a discrete and co-compact sub-group of affine isometries of $E^n$, then the sub-group 
$\Lambda$ of the elements of $\Gamma$ that are translations is a normal sub-group of finite index in $\Gamma$ and a lattice of $E^n$.\\
ii) The number of crystallographic groups in dimension $n$ is finite up to an isomorphism. Two crystallographic groups on $E^n$ are isomorphic
if and only if they are conjugate by an element of the affine group.
\end{theorem}

This shows that we have the following exact sequence of groups
$$
0\longrightarrow\Lambda\longrightarrow\Gamma\longrightarrow G\longrightarrow e
$$

where $G$ is finite.\\

The manifold $E^n/\Gamma$ admits a Riemannian covering by the flat torus $E^n/\Lambda$.
The group $G$ is the \emph{holonomy group} of the manifold. It admits a natural realization as
a subgroup of $O(n)$.\\

There are two types of Bieberbach groups in dimension $2$, $10$ in dimension $3$ and $74$ in dimension $4$. Obtaining a precise classication in any dimension seems difficult 
and maybe hopeless. Some classes have been studied by Charlap, see \cite{charlap}, and\ more recently
by  Szczepa\'nski, cf. \cite{szz}.

\smallskip
Bieberbach manifolds are known to be boundaries (see
\cite{gordon}  and \cite{hamroy}).
Here we address the more precise (and indeed very different) 
question of realizing a Bieberbach manifold as 
the \emph{totally geodesic boundary} of a compact flat manifold. In dimension $1$, the answer, although well-known, is not completely
trivial, since it involves the M\"obius band. 
The general situation, as the main result of this paper shows (see Theorem \ref{maintheo}) ,
is basically the same: an $n$-dimensional compact flat manifold with geodesic boundary is either a product
$I\times M$, where $M$ is an $(n-1)$ dimensional Bieberbach manifold and $I$ a compact interval (if the boundary is not connected) 
or a non-tivial Riemannian $I$-bundle over an $(n-1)$-dimensional Bieberbach $M^\prime$ (if the boundary is connected) ;
this boundary is a non-trivial $\ZZ_2$-bundle over $M^\prime$. In particular it admits a fixed-point free isometric involution. In the rest of the paper we give some applications of this  result.\\

In section $3$, we give first a classification of 3-dimensional flat manifolds with geodesic boundary;
those that have a connected boundary are somewhat the 3-dimensional M\"obius bands. There are $3$ types of such manifolds. In the same way as the Klein bottle can be realized by gluing two M\"obius bands, some Bieberbach manifolds can be obtained from them by a suitable gluing, see \ref{recollement}. Then, using Theorem \ref{maintheo} we specify, among the ten types (up to an affine diffeomorphism) of 3-dimensional Bieberbach manifolds those that are the geodesic boundaries of a compact flat manifold. This enables us to give a classication of flat compact $4$-manifolds with connected geodesic boundary (cf. \ref{list4}).       

In section $4$, we show that Bieberbach  manifolds whose  holonomy group is isomorphic to $\ZZ_{2m+1}$ or $\ZZ_2$  can be realised
as a geodesic boundary of a compact flat manifold.  

In section $5$, we study the 4-dimensional manifolds with a holonomy group isomorphic to $\ZZ_2^3$ (Generalized Hantzsche-Wendt manifolds) 
and give lowest-dimensional  examples of non-orientable Bieberbach manifolds that are not a geodesic boundaries of a compact flat manifold.

\subsection{Notations} All the manifolds we  consider are  supposed to be connected if the contrary is not explicitely stated.
We shall identify $E^n$
with $\RR^n$ equipped with its standard orthonormal basis.

 The translation of vector $\vec{v} \in \RR^n$ is denoted $t_v$ . 
If $V$ is an affine  subspace of $\RR^n$, and  $a$ a  vector parallel  to $V$, we  denote by $\sigma_{V,a}$ the glide reflection
obtained by composing the orthogonal  reflection with respect to  $V$ and the translation  $t_a$. 
If $V$ is a line, we shall content ourselves with mentioning the length of $a$.
 
The flat rectangular torus whose fundamental domain is a rectangle of sides $a$ and $b$ will be denoted $T_{a,b}$,
and the flat Klein bottle whose fundamental group is generated by  $(x,y)\mapsto (x+a/2,-y)$ and $(x,y)\mapsto (x,y+b)$
by $K_{a,b}$.

\section{The structure of flat manifolds with geodesic boundary} \label{mainsect}
\begin{theorem}

The universal (Riemannian) cover of a compact $n$-dimensional flat Riemannian manifold with geodesic boundary is isometric to
$I\times \mathbb{R}^{n-1}$.  
\end{theorem}

\begin{proof}
Let $M$ be such a manifold. In a neighborhood of any component $S$ of the boundary, the metric   can  be written $g_S+dt^2$ ($g_S$ is the induced metric
and $t$ the distance to $S$. Therefore by gluing two copies of $M$ along the boundary, we get a compact flat boundaryless manifold $N$. Let $p:\RR^n\rightarrow N$ its universal Riemannian covering. The inverse image $p^{-1}$ of the boundary is a (disjoint) union of complete totally geodesic hypersurfaces of $\RR^n$ that is to say of parallel hyperplanes.\\

In fact, there are infinitely many such hyperplanes:  indeed, the deck transforms of the
covering  $p:\RR^n\rightarrow N$ act on this family of hyperplanes,
and form a Bieberbach group in dimension $n$, which contains a lattice of rank $n$. Now, $p^{-1}(M)$ is a disjoint union of bands limited by  two such hyperplane. Let $B\simeq I\times \RR^{n-1}$ be one of them. The restriction of $p$ to $B$ gives a universal cover of $M$.
\end{proof}

\bigskip

Now, we study the manifolds themselves. We can suppose for simplicity that $I=[-1,1]$.\\

The fundamental group $\Gamma$  of $M$ , viewed as a sub-group of $\mathrm{Isom}(\mathbb{R}^n)$, 
  leaves the boundary  $\{1\}\times \mathrm{R}^{n-1}\cup  \{-1\}\times \mathrm{R}^{n-1}$ globally invariant and admits
 a compact fundamental  domain. The hyperplane $P=\{0\}\times \mathrm{R}^{n-1}$ is also globally invariant,
 and the restriction to this hyperplane of the action of $\Gamma$ still admits a fundamental domain. Two cases are possible.\\
 %

\begin{enumerate}

\item Each  connected  component of the  boundary of $I\times\RR^{n-1}$ is $\Gamma$-invariant.  Then, any parallel hyperplane
to these components is also $\Gamma$-invariant. In this situation,  $\Gamma$  is isomorphic to its restriction to
$\RR^{n-1}$, which is a $(n-1)$-dimensional Bieberbach group that we denote $\Gamma'$. Then $M$ is isometric to $I\times(\RR^{n-1}/\Gamma')$. Its boundary has two connected components and can be identified with $\RR^{n-1}/\Gamma'\sqcup \RR^{n-1}/\Gamma'$.\\

\item There are elements of $\Gamma$ which exchange the components. Then the subgroup $\Gamma_0$ of $\Gamma$  which preserves these components is a normal sub-group of index $2$. We have the following property, which is trivial in case i).

\end{enumerate}

\begin{lemma} The restriction map $\gamma \mapsto \gamma_{\mid P}$ is injective.
\begin{proof}
 Suppose that $\gamma_{\mid P}$ is the identity map. If $\gamma\in\Gamma_0$, clearly $\gamma=Id_\Gamma$.
 If $\gamma\in\Gamma\setminus\Gamma_0$ then $\gamma$ must be orthogonal reflection with respect to $P$.
 This is impossible, since $\Gamma$ acts freely.
 \end{proof}
\end{lemma}
 
Therefore the restriction to $P$ maps $\Gamma$ and $\Gamma_0$ isomorphically onto
Bieberbach groups in dimension $(n-1)$, that we shall denote $\Gamma^\prime$ and $\Gamma_0^\prime$.

Denoting by  $(x,y)$  the elements of $\RR\times\RR^{n-1}$,
the  actions of $\Gamma $   on $\RR^n$ and of  $\Gamma^\prime$ on $\RR^{n-1}$ are related
as follows 
$$ \gamma\cdot (x,y)=\left\{\begin{array}{rl} (x,\gamma'\cdot y)&\mathrm{si}\ \gamma\in\Gamma_0\\
(-x,\gamma'\cdot y)&\mathrm{si}\ \gamma\notin\Gamma_0\end{array}\right.$$
The boundary of $M$ is connected and isometric to 
$\RR^{n-1}/\Gamma_0'$, and  $M$ admits a Riemannian covering of order $2$
isometric to  $I\times\RR^{n-1}/\Gamma_0$. This situation is exactly what happens in dimension $2$
with the M\"obius band.\\

\textbf{Remark.}  There is of course a strong contrast with the hyperbolic case. In dimension $2$,
if a compact hyperbolic manifold has a (non-empty) geodesic boundary, this boundary has
at least $3$ components. For higher dimension, see \cite{fp}.

\medskip

A consequence of this discussion is a characterisation of connected geodesic boundaries of Bieberbach manifolds.

\begin{theorem}\label{maintheo}
 For a Bieberbach manifold $\RR^{n-1}/\Gamma_0$ the following properties are equivalent.
 \begin{enumerate}
\item It can be realized as the totally geodesic boundary of a compact flat $n$-dimensional manifold.

\item It admits a fixed point free isometric involution.
\item
There is an exact sequence
 $$
1\longrightarrow \Gamma_0\longrightarrow  \Gamma\longrightarrow \mathbb{Z}_2\longrightarrow 1$$
where $\Gamma$ is a $(n-1)$-dimensional Bieberbach group.
 \end{enumerate}
\end{theorem} 
\begin{proof}

We have just seen that i) implies iii). If iii) is true, then $\Gamma$ normalises $\Gamma_0$
and defines an isometric involution, which is fixed-point free since $\Gamma$ is a Bieberbach group.
Now, let $\sigma$ be such an isometry of
 $\RR^{n-1}/\Gamma_0$. Take the manifold
 $[-1,1]\times \RR^{n-1}/\Gamma_0$ equipped with the product metric.
 The map $(x,y)\mapsto (-x, \sigma(y))$ is a fixed point free isometric involution, and the quotient
 manifold has the required property.
\end{proof}

\begin{rem}\label{involution} This theorem allows to check whether a given Bieberbach $\RR^{n-1}/\Gamma_0$ manifold is (or is not) the geodesic boundary of an $n$-dimensional compact flat  manifold. One must prove (or disprove) the existence
of a $\gamma \in \mathrm{Isom}(\RR^{n-1})$ with the following properties:

\begin{enumerate}
\item  $\gamma$ normalizes $\Gamma_0$ (it must go to the quotient as an isometry);
\item $\gamma^2 \in \Gamma_0$ (involutiveness) ;
\item The sub-group of $\mathrm{Isom}(\RR^{n-1})$ generated by $\gamma$ and $\Gamma_0$ acts freely.

 \end{enumerate}
 
 An involutive fixed points free isometry which is not a translation doubles the order of the holonomy group.
 Therefore, the case where the holonomy group has maximal order in a given dimension is more simple, since it would be enough to only consider the translations $\gamma \in \Lambda/2 \setminus \Lambda$.

\end{rem}

\begin{defi} An isometry of a Bieberbach manifold is called \emph{admissible} 
if it is involutive and does not admit fixed points.

\end{defi}

To summarize this discussion, a compact flat $n$-dimensional  manifold with connected geodesic boundary is a non trivial $I$-bundle
over a $(n-1)$-dimensional Bieberbach manifold (that one could call the \emph{soul}) and the boundary is a two-fold non-trivial covering of the soul.

\section{Two and three dimensional boudaries }\label{examples}
\subsection{Two dimensional boundaries}\label{dim2}
We will study flat manifolds with geodesic boundary in dimension $3$. We follow the notations of the previous section.\\

In the case i), i.e., if each connected component of the boundary of $I\times\RR^{2}$ is stable by $\Gamma$, then $M$ is the product of a flat torus or a flat Klein bottle by a segment.\\

In the case ii), i.e., if the sub-group of $\Gamma$ which keeps stable the connected components of the boundary of $I\times\RR^{2}$ is a sub-group $\Gamma_0$ of index $2$,  then $\Gamma'$ is either a lattice of the Euclidean plane or the fundamental group of a Klein bottle.
\begin{itemize}

\item If $\Gamma'$ is a lattice, then its sub-groups of index $2$ are of the form $\ZZ v+\ZZ v^\prime$ for a suitable basis
$(v,v^\prime)$
of $\Gamma$.
The manifold is the quotient of $[-d,d]\times\RR^2$ by the sub-group of $\mathrm{Isom}(\RR^3)$
generated by $t_v$ and the glide reflection $\sigma_{P,\frac{v^\prime}{2}}$, where $P$ is the plane $x=0$ with $v$ and $v'$ parallel to $P$.
This manifold is not orientable; its boundary is the flat torus
$\RR^2/\ZZ v+\ZZ v^\prime$, whereas the plane $P$ gives in the quotient
the flat torus $\RR^2/\ZZ v+\ZZ \frac{v^\prime}{2}$.
The boundary and the soul are both tori.
The similarity with the M\"obius band is such that we suggest to call this manifold the \emph{solid M\"obius band}. 
We denote it by $TT_{v,v^\prime,d}$. Note that such a manifold exists in any dimension.

\item If $\Gamma'$ is the fundamental group of the Klein bottle $K_{a,b}$, generated by the glide reflection $(y,z)\mapsto (y+\frac{a}{2},-z)$
and the translation $(y,z)\mapsto (y,z+b)$, there are two types of sub-groups of index $2$.\\
\begin{itemize}
 \item

First, $\Gamma_0$ can be the lattice associated to $\Gamma$.
The manifold is then the quotient of $[-d,d]\times\RR^2$ by the sub-group of $\mathrm{Isom}(\RR^3)$
generated by
$t_{(0,0,b)}: (x,y,z)\mapsto (x,y,z+b)$ and $\sigma_{D,(0,\frac{a}{2},0)}:(x,y,z)\mapsto  (-x,y+\frac{a}{2},-z)$
where $D$ is the straight line $x=z=0$. This manifold is orientable, and its boundary is the torus $T_{a,b}$.
The boundary is a rectangular torus and the soul is a Klein bottle. We denote it by $TK_{a,b,d}$.
 
\item

Finally, $\Gamma_0$ can be generated by the glide reflection of displacement
$\frac{a}{2}$ and the translation
$(x,y,z)\mapsto (x,y,z+2b)$.
The manifold is then the quotient of $[-d,d]\times\RR^2$ by the sub-group of $\mathrm{Isom}(\RR^3)$
generated by
$\sigma_{P,(0,0,b)}: (x,y,z)\mapsto (-x,y,z+b)$ and $ \sigma_{P',(0,\frac{a}{2},0)}(x,y,z)\mapsto (x,y+\frac{a}{2},-z)$
where $P$ and $P'$ are respectively the planes $x=0$ et $z=0$. 
This manifold is not orientable; the  boundary and the soul are both  Klein bottles.
We denote it by $KK_{a,b,d} $.

\end{itemize}
\end{itemize}

Therefore, we have three  types of compact flat manifolds with connected geodesic boundary. They can be thought of as 3-dimensional Mobius bands.
The following table summarizes this discussion.\\

\begin{table}[htbp]\renewcommand{\arraystretch}{1.2}
\begin{center}
\begin{tabular}{|c|c|c|}
\hline
 Generators of $\Gamma$ & Orientability & Boundary \\
\hline
$t_v$ and $\sigma_{P,\frac{v^\prime}{2}}$ & non-orientable & $T^2=\RR^2/\ZZ v+\ZZ v^\prime$\\
\hline
$t_{(0,0,b)}$ and $\sigma_{D,(0,\frac{a}{2},0)}$ & orientable & $T_{a,b}$\\
\hline
 $\sigma_{P,(0,0,b)}$ and $\sigma_{P',(0,\frac{a}{2},0)}$ & non-orientable & $K_{a,2b}$\\
\hline
  \end{tabular}\end{center}
\end{table}

By gluing two such manifolds along isometric boundaries, some $3$-dimensional Bieberbach manifolds can be obtained. We will give the details in \ref{recollement} after giving a precise description of these manifolds.

\subsection{Three dimensional boundaries}\label{dim3}
Three dimensional Bieberbach manifold have been classified by Hantzsche and Wendt in 1935 (See \cite{wolf} 
Th. 3.5.5 (orientable case) and Th. 3.5.9 (non orientable case)). 
There exist ten compact and flat manifolds of dimension $3$ up to an affine diffeomorphism, six are orientable and four are non-orientable.

In the orientable case, these types are caracterized by the holonomy group $G$. Besides flat tori, they are as follows
(we use the notations of Thurston, cf. \cite{th}, 4.3.
\begin{itemize}
 \item  
The manifold $C_2$: the group $\Gamma$ is generated by two orthogonal translations $\{a_1,a_2\}$ and a 
glide reflection $\sigma_{L,a_3/2}$, where the line $L$ is orthogonal to $a_1$ and $ a_2 $. Its holonomy group is isomorphic to $\mathbb{Z}_2$.

\item The manifold $C_3$: the group $\Gamma$ is generated by independent translations $\{a_1,a_2\}$ which generate a flat hexagonal lattice and 
a srew motion  $(R,t_{a_3/3})$ of angle $2\pi/3$ around a vector $a_3$ orthogonal to $a_1$ and $a_2$.
 Its holonomy group is isomorphic to $\mathbb{Z}_3$.
\item
The manifold $C_4$: the group $\Gamma$ is generated by two translations $\{a_1,a_2\}$ which generate a flat square lattice and a screw motion
$(R,t_{a_3/4})$ of angle $\pi/2$ around a vector $a_3$ orthogonal to $a_1$ and $a_2$. Its holonomy group is isomorphic to $\mathbb{Z}_4$.

\item The manifold $C_6$: the group $\Gamma$ is generated by two translations $\{a_1,a_2\}$ which generate a flat hexagonal lattice and 
a screw motion  $(R,t_{a_3/6})$ of angle $\pi/3$ around a vector $a_3$ orthogonal to $a_1$ and $a_2$. Its holonomy group is isomorphic to $\mathbb{Z}_6$.

\item The manifold $C_{2,2}$: let $L$, $M$ and $N$ be three disjoint lines pairwise orhogonal. The group
$\Gamma$ is generated by the glide reflections $\sigma_{L, 2d(M,N)}$,  $\sigma_{M, 2d(N,L)}$ and  $\sigma_{N, 2d(L,M)}$
(see the very suggestive figure of \cite{th}, p. 236; in fact two reflections suffice).
Its holonomy group is isomorphic to $\mathbb{Z}_2\times \mathbb{Z}_2$.

\end{itemize}

\begin{prop}\label{orient} Among orientable $3$-dimensional Bieberbach manifolds, only the torus, the manifold $C_2$ and the manifold $C_3$ are the geodesic boundary of a compact flat $4$-dimensional manifold.
\end{prop}

\begin{proof}

For the manifold $C_2$, the translation $a_1/2$ (and also $a_2/2$) goes to the quotient and does not admit fixed points. 
It is therefore a geodesic boundary. Note that $C_2$ covers $C_4$ if the corresponding $\{a_1,a_2\}$ lattice is square. 
The manifold $C_3$ is also a geodesic boundary since it always covers the manifold $C_6$.
Also, the translation  $a_3/2$ goes to the quotient and does not admit fixed points (this is a particular
case of Theorem \ref{impair}) .\\

The manifolds $C_4$, $C_6$ and $C_{2,2}$ do not admit any admissible involution.
Indeed, their holonomy group, viewed as a sub-group of $O(3)$, is maximal.
Therefore, an isometric involution must be given by an element of $\Lambda/2\setminus \Lambda$.
For $C_4$ and $C_6$, such an involution is given by $t_{a_3/2}$ but it has fixed points.

The situation is the same for $C_{2,2}$: the sub-group of $\mathrm{Isom}(\RR^3)$ generated by
$\Gamma$ together with a translation in $\Lambda/2\setminus\Lambda$ has fixed points.

\end{proof}

\smallskip

\begin{rem}.
We have seen in the proof of the previous proposition (prop. \ref{orient}) that there are sometimes two ways to obtain the manifolds $C_2$ and $C_3$ as totally geodesic boundaries.
In fact, the manifold $C_2$ (if the corresponding lattice is orthogonal) and the manifold $C_3$
 admit two types of admissible isometries. The first type belongs to the neutral component of the isometry group. The second type, i.e., fixed point free involutions which do not belong to the neutral component of the isometry group, give nice examples of compact $4$-dimensional manifolds with connected geodesic boundary, namely one with soul $C_4$ and boundary $C_2$, and one with soul
$C_6$ and boundary $C_3$.
\end{rem}

Now, we study the four types of non-orientable Bieberbach manifolds in dimension three. Each one of them can be obtained with a suspension of
the Klein bottle (cf. \cite{elaf08}  for  details). They are as follows (we  use the notations of Wolf, cf. \cite{wolf}, ch.3)
as we did in \cite{elaf08}

\begin{itemize}  

\item The manifold $B_1$: the group $\Gamma$ is generated by two orthogonal translations $\{a_1,a_2\}$  and a glide 
reflection  $\sigma_{P,a_3/2}$ with respect to a plane $P$ generated by $\{a_2,a_3\}$ and where $a_3$ is orthogonal 
to $a_1$. Its holonomy group is isomorphic to $\mathbb{Z}_2$.\\
The lattice $\Lambda$ corresponding to $B_1$ is generated by the basis $\{a_1,a_2,a_3\}$. 

\item The manifold $B_2$: the group $\Gamma$ is generated by two glide reflections $\sigma_{P_1,a_1/2}$ and $\sigma_{P_2,a_2/2}$ with respect to
two parallel planes $P_1$ and $P_2$, where $a_1$ and $a_2$ are linearly independent vectors parallel to these planes.
Its holonomy group is isomorphic to $\mathbb{Z}_2$.
The lattice $\Lambda$ corresponding to $B_2$ is generated by $\{a_1,a_2,a_3\}$ where $a_3=\frac{a_1+a_2}{2}+2d$ and 
$d$ is the vector (orthogonal to $P_i$) that sends $P_1$ to $P_2$. 

\item The manifold $B_3$:
$\Gamma$ is generated by a translation $a_3$, a glide reflection $\sigma_{P,a_1/2}$ with respect to
 a plane $P$ perpendicular to $a_3$ and generated by two orthogonal vectors $\{a_1,a_2\}$ and a screw motion $(R,t_{a_2/2})$
of angle $\pi$ around a straight line $L \subset P$ and parallel to $a_2$. Its holonomy group is isomorphic to $\mathbb{Z}_2\times \mathbb{Z}_2$.

\item The manifold $B_4$: the group $\Gamma$ is generated by a glide reflection $\sigma_{P,a_1/2}$ with respect to a plane $P$ generated by two orthogonal vectors $\{a_1,a_2\}$, a screw motion $(R,t_{a_2/2})$ of angle $\pi$ around a straight line $L$ parallel to $a_2$ but not contained in $P$. Its holonomy group is isomorphic to $\mathbb{Z}_2\times \mathbb{Z}_2$. \\
The lattice $\Lambda$ corresponding to $B_4$ is generated by the orthogonal basis$\{a_1,a_2,a_3\}$ where $|a_3|=4dist(P,L)$.

\end{itemize}

It should be noted that the orientable cover of $B_1$ and $B_2$ is a torus,
whereas the orientable cover of 
$B_3$ and $B_4$ is $C_2$ (with an orthogonal lattice).

\begin{prop}\label{nonorient} Every non-orientable $3$-dimensional Bieberbach manifolds is the geodesic boundary of a compact flat $4$-dimensional manifold.
\end{prop}

\begin{proof}

The manifold $B_1$ admits a self-cover: the translation $a_1/2$ goes to the quotient and does not admit fixed points. It also covers the manifolds $B_3$ and $B_4$ if the lattice $\Lambda$ is orthogonal. The manifold $B_2$ also admits an admissible isometry: the translation $2d$ (and also the translation $\frac{a_1+a_2}{2}$, see section \ref{Z2} below) goes to the quotient and does not admit fixed points. Note that the quotient of $B_2$ by the translation $2d$ is a manifold of type $B_1$.\\

The manifolds $B_3$ admits a self-cover: the translation $a_3/2$ goes to the quotient and does not admit fixed points. The manifold $B_4$ also admits an admissible isometry: the translation $a_3/2$ goes to the quotient and does not admit fixed points. Note that the quotient of $B_4$ by the translation $a_3/2$ is a manifold of type $B_3$.

\end{proof}

As a consequence of that discussion, we have the following

\begin{cor}\label{list4} Up to diffeomorphisms, there are fourteen types of compact flat $4$-\ with connected boundary. Namely
 
 \begin{itemize}
  \item Self-coverings: $(T^3,T^3),\,(C_2,C_2),\, (B_1,B_1),\,(B_3,B_3) $;
  \item orientation coverings:  $(T^3,B_1),\,(T^3,B_2),\,(C_2,B_3),\,(C_2,B_4)$;
  \item misceallenous: $(T^3,C_2),\,(C_2,C_4),\,(C_3,C_6),\,(B_1,B_3)$,
  
  $(B_1,B_4), \, (B_2,B_1),\,(B_4,B_3)$
 \end{itemize}
(in that list have denoted by $(A,B)$ the total  space of a  $\ZZ_2$-fibration of $A$ over $B$).
Morever, if two such manifolds are diffeomorphic, they are affinely diffeomorphic. 
\end{cor}
\begin{proof}
Many cases have already been covered. A systematic use of remark \ref{involution} gives the
admissible involutions we did not encountered yet. They concern $T^3$ and $B_1$, and are
suitable glide reflections with respect to a line.

\end{proof}

\subsection{Gluing along the boundary}\label{recollement}
It is well known that the Klein bottle can be realized by gluing two isometric M\"obius bands  along their boundary. A similar
phenomena occurs in higher dimension. 
\begin{itemize}
 \item 
  $B_1$ can be obtained by gluing two copies of  $TT_{v,v^\prime,d}$;
 \item 
 $B_2$ can be obtained by gluing $TT_{v,v^\prime,d}$ and $TT_{v^\prime,v,d}$; 
 \item 
 $B_3$ can be obtained by gluing two copies of $KK_{a,b,d}$;
 \item
  $B_4$ can be obtained by gluing $TT_{a,b,d}$ and $TK_{a,b,d}$;
  \item
  $C_2$ can be obtained by gluing two copies of $TK_{a,b,d}$;
  \item
 $C_{2,2}$ can be obtained by gluing $TK_{a,b,d}$ and $TK_{b,a,d}$.

\end{itemize}
Te see that, we proceed backward: we check that $B_1$ and $B_2$ admit totally geodesic foliations by
tori, with 2 
exceptional leaves which are tori of ``half-size''.  
Indeed, both fundamental groups contain glide reflections. Foliate $\RR^3$ by parallel planes to
the reflection plans 
and go to the quotient.

Concerning $B_3$, the same procedure gives a foliation by Klein bottles, with two exceptional leaves which are Klein bottles
of half size.

The procedure is still the same for $B_4$, but the situation is more involved: we also obtain a foliation whose generic leaves are tori,
but there are 2 exceptional leaves, namely a (half-size) torus
and a Klein bottle.

Concerning $C_{2,2}$, we just foliate $\RR^3$ with orthogonal planes to the axis of a screw motion of the
fundamental group and  go to the quotient. This time we get a foliation  whose generic leaves  are tori, with
$2$ exceptional leaves which are (generally non isometric) Klein  bottles.

\section{Bieberbach manifolds with particular Holonomy groups} In this section, we will give a general result for Bieberbach manifolds with a holonomy group isomorphic to $\ZZ_{2m+1}$ or $\ZZ_2$: they all can be realised as the geodesic boundary of a compact flat manifold. 

\subsection{Manifolds with cyclic Holonomy group of odd order}

\begin{theorem}\label{impair} Every Bieberbach manifolds with a cyclic holonomy group of odd order is the totally geodesic boundary of a compact flat manifold.

\end{theorem}

\begin{proof} We will show that such a manifold admits an admissible translation. Let $\gamma$ be a lift in $\Gamma$ of a generator of
the holonomy group (supposed to be of order $N$). Its  linear part must fix some non-trivial vector space $V$, and $\gamma^N$
is a non-trivial translation $t_a$, where $a\in V$.
Then, using oddness, we see that the translation $t_{a/2}$ does not admit fixed points.
\end{proof}

\subsection{Manifolds with $\ZZ_2$ holonomy group}\label{Z2}
Suppose that $\RR^n/\Gamma$ has holonomy $\ZZ_2$. Its realization as a subgroup  of $O(n)$
is just $\{I,\sigma_P\}$, where $\sigma_P$ is an orthogonal reflexion
with respect to a $k$-dimensional subspace $P$.
The lattice $\Lambda$ has index $2$ in $\Gamma$, and its non trivial class  $\mathcal{R}$
is a set of glide-reflections of type $\sigma_{Q,a}$, where 
the $k$-dimensional affine subspaces $Q$ are parallel to $P$. It is clear that if $\sigma_{Q,a}\in\mathcal{R}$, then $2 a\in\Lambda$.\\

Now, for $\sigma_{Q_1,a_1}$ and  $\sigma_{Q_2,a_2}$ in $\mathcal{R}$, let $d$ be the translation orthogonal to the $Q_i$ which sends $Q_1$ onto $Q_2$.

\begin{lemma} \label{glide} We have the following properties:

\begin{enumerate}
  \item the vector $a_1+a_2+2d$ belongs to $ \Lambda$.
  
 \item  the set of vectors $d$ is a lattice of dimension $(n-k)$.
\end{enumerate}
 
\end{lemma}
\begin{proof}
 The first claim follows from the fact that $\sigma_{Q_1,a_1}\circ \sigma_{Q_2,a_2}=a_1+a_2+2d$. 
 Now, we have a transitive action (by translations) of $\Lambda/2$ on the set of
 $k$-planes $Q$ such that $\sigma_{Q,a}\in \mathcal{R}$. Pick one of them
 $Q_0$. The third claim amounts to say that the image of $\Lambda+ Q_0$ is
 a lattice in $\RR/Q_0$. It is clearly cocompact.
 Now, if  $a_1+a_2+2d\in \Lambda$ and $2(a_1+a_2)\in\Lambda$, then $4d\in\Lambda$,
 which proves discreteness.
 
\end{proof}
 
 \begin{theorem}
  If an $n$-dimensional Bieberbach manifold has holonomy $\ZZ_2$, then it is a geodesic boundary.
  \end{theorem}

\begin{proof}

 The case when $\sigma_P$ is a reflection with respect to a line parallel to a vector $a_0 \in \Lambda$ is completely analogous to
 that of the manifold $C_2$, which we described in \ref{dim3}. The lattice $\Lambda$ is an orthogonal 
 sum $\Lambda_0\bigoplus\ZZ a_0$, and $\Gamma\setminus\Lambda$ is composed with glide reflections
 $\sigma_{D,a}$, where the vectors $a$ are odd integer multiples of $\frac{a_0}{2}$.
 The lines $D$ are translated of one of them by a vector of $\Lambda_0/2$.
 Any translation of  $\Lambda_0/2\setminus \Lambda$ goes to the quotient as a fixed
 point free isometric involution of $\RR^n/\Gamma$.
 
 \smallskip
 
 Now, we consider the case where $\sigma_P$ is a reflection with respect to a $k$-dimensional plane, where $k\geq 2$.
 For an affine $k$-plane $Q$, set
 $$
 \mathcal{T}_Q = \{v\in\RR^n,  \sigma_{Q,v}\in \Gamma\}
 $$

Lemma \ref{glide} shows that the affine $k$-planes $Q$ such that $\mathcal{T}_Q$ is not empty
are parallel. Following the notations
of the previous lemma, the distance of two neighboring $k$-planes is $\vert d\vert$,
and the vector $4d$ is in $\Lambda$. 
Therefore, by composing a glide-reflection with a translation $t_{4d}$,
we see that for two $k$-planes whose distance is $2\vert d\vert$ the sets  $\mathcal{T}_Q$ are the same.
There are two possibilities.

\begin{enumerate}
 \item  $\mathcal{T}_Q $ is always the same for all the $Q$'s in $\mathcal{R}$,
 and can be written as $a/2 + \Lambda_0$, where $a\in\Lambda$
 and $\Lambda_0$ is the orthogonal projection of $\Lambda$ onto $P$.
 In this case, $\Lambda$ itself is the orthogonal sum $\Lambda_0\bigoplus 2\ZZ d$.
 The situation is the exact analogue of that of the manifold $B_1$ in dimension $3$.
 
 \item there are exacly two possibilities for  $\mathcal{T}_Q$, namely
 $a_1/2+\Lambda_0$ and $a_2/2+\Lambda_0$, where $a_1$ and $a_2$ are independent and belong to $\Lambda$.
 The lattice $\Lambda$ is the (not orthogonal!) sum
 $\Lambda_0\bigoplus\ZZ(\frac{a_1+a_2}{2}+2d)$. This is the situation of $B_2$ in dimension $3$.
 
\end{enumerate}

The translation $t_{d}$ in the first case, $t_{2d}$ in the second case, gives by going to the quotient
a fixed point free isometric involution.
\end{proof}

\section{Hantzsche-Wendt manifolds of dimension $4$}\label{Han-Wen}

In this section, we shall study Generalized Hantesche-Wendt n-manifolds, i.e., Bieberbach n-manifolds with holonomy isomorphic to $\mathbb{Z}^{n-1}$. In dimension $3$, there are three Generalized Hantezsche-Wendt manifolds. The manifolds $B_3$ and $B_4$ are a geodesic boundary whereas the manifold $C_{2,2}$ is not. This shows that the data of the holonomy group of a Bieberbach manifold is generally not sufficient to decide whether it satisfies that property.\\


In dimension $4$, there are $12$ Generalized Hantesche-Wendt manifolds up to an affine isometry. We shall  determine which among these manifolds are or are  not geodesic boundaries of a compact flat manifold. To see that, a rough description of these manifolds is necessary. In what follows,
we view $\RR^4$ as the Euclidean spaces of dimension $4$ equipped with its  standard orthonormal basis. All these manifolds are quotients of flat tori $\RR^n/\Lambda$, where $\Lambda$ is the orthogonal lattice generated by $ae_1, be_2, ce_3, de_4$.\\

In our description, we use an affine isometries of the Euclidean space which clearly go to the quotient as isometries of the torus $\RR^n/\Lambda$, and will abusively use the same notation for an affine isometry and the corresponding quotient map. The $4$-dimensional Hantzsche-Wendt manifolds will be then the quotient of $\RR^n/\Lambda$ by the generators (affine isometries) given in the table below.

\textbf{Notations:} A symmetry with respect to the hyperplane  $yzt$ will be denoted by $S_x$ (similarly for symmetries with respect to $xyt$, $xyz$ and $xzt$).\\

A symmetry with respect to the plane $zt$ will be denoted by $S_{xy}$ (similarly for symmetries with respect to $xz$, $xt$, $yz$, $yt$ and $zt$).\\

A translation of vector $\frac{a}{2}e_1$ will be denoted by $T_x$ (similarly for translations $\frac{b}{2}e_2$, $\frac{c}{2}e_3$ and $\frac{d}{2}e_4$). The composition of a translation $T_x$ and a translation $T_y$ will be denoted $T_{xy}$.

\smallskip

\begin{table}[htbp]\renewcommand{\arraystretch}{1.2}
\begin{center}
\begin{tabular}{|c|c|c|c|}
\hline
 H-W Manifold & Generators & H-W Manifold & Generators \\
\hline
$H-W_{1}$ & $T_{yt}S_z,\ T_zS_{xt},\ T_xS_{yt}$ & $H-W_{2}$ & $T_tS_z,\ T_{yz}S_{xt},\ T_xS_{yt}$ \\
\hline
$H-W_{3}$ & $T_tS_z, \ T_{yz}S_{xz},\ T_{x}S_{yz}$ & $H-W_{4}$ & $T_{yz}S_z,\ T_{yz}S_{xz},\ T_xS_{yz}$\\
\hline
 $H-W_{5}$ & $T_yS_z,\ T_tS_{xz},\ T_xS_{yz}$ & $H-W_{6}$ & $T_{yz}S_z, \ T_{t}S_{xz}, \ T_xS_{yz}$\\
\hline
 $H-W_{7}$ & $T_yS_z,\ T_{yt}S_{xz},\ T_xS_{yz}$ & $H-W_{8}$ & $T_{xz}S_z, \ T_{yt}S_{xz}, \ T_xS_{yz}$\\
\hline
 $H-W_{9}$ & $T_{yz}S_z,\ T_{yt}S_{xz},\ T_xS_{yz}$ & $H-W_{10}$ & $T_{xyz}S_z, \ T_{yt}S_{xz}, \ T_xS_{yz}$\\
\hline
 $H-W_{11}$ & $T_yS_z,\ T_{yzt}S_{xz},\ T_xS_{yz}$ & $H-W_{12}$ & $T_{yz}S_z, \ T_{yzt}S_{xz}, \ T_xS_{yz}$\\
\hline
  \end{tabular}\end{center}
\end{table}

Now, we come to the main result of this section.

\begin{theorem}
 Among $4$-dimensional generalized Hantzsche-Wendt manifolds,
 only four of them (all non-orientable) are not totally geodesic boundaries, namely the manifolds with vanishing Betti number $H-W_1$ and $H-W_2$ and the manifolds $H-W_3$ and $H-W_4$.
\end{theorem}

\begin{proof} In dimension $4$, there does not exist a Bieberbach manifold with a holonomy group of order $16$. Therefore, it would be enough to check, for each $H-W_i$, if it admits an admissible translation.\\

For the manifolds $H-W_i$, where $i\in \{5,6,7,8,9,11,12\}$, the translation $T_z$ is admissible. For the manifolds $H-W_{10}$, The translation $T_y$ is admissible.\\

Now, for the manifolds $H-W_{i}$, $i=1,2,3$ or $4$, direct inspection shows that, in the four cases we have to consider, a translation in $\Lambda/2 \setminus \Lambda$ that goes to the quotient always admits fixed points.

 \end{proof}

As a conclusion, it seems that if a Bieberbach manifold has an holonomy group which is a product of $\ZZ_2$, the situation becomes very arbitrary even if the order of the holonomy becomes maximal (in the class of the manifolds whose holonomy is a product of $\ZZ_2$) in a given dimension.
 
\bigskip

\emph{Acknowledgements:} The authors would like to thank Andrezj Szczepa\'nski for giving the description of 4-dimensional Bieberbach manifolds with zero first Betti number by the use of the useful package CARAT.

\smallskip


\begin{thebibliography}{99}

\bibitem{fourcrys}  Brown, H; Bulow, R.; Neubuser, J.; Wondratschek, H.; Zassenhaus, H; Crystallographic Groups of four-Dimensional space, 1st edition, Wiley, New York (1978).


\bibitem{charlap} Charlap, L.S.; Bieberbach Groups and Flat Manifolds, Springer Universitext, Berlin (1986).


\bibitem{elaf08} Elmir, C.; Lafontaine, J.; Sur la G\'eom\'etrie Systolique des Vari\'et\'es de Bieberbach, Geom. Dedicata. 136, 95--110 (2008).

\bibitem{elaf13} Elmir, C.; Lafontaine, J.; The Systolic Constant of Orientable Bieberbach $3$-manifolds, Ann. Math. Toulouse, Sér. 6 Vol. 22 no. 3, 623--648 (2013).

\bibitem{fp} Frigerio, R.; Petronio, C.; Construction and Recognition of Hyperbolic $3$-manifolds with Geodesic Boundary, TAMS. 13, 171--184 (2001).





\bibitem{ghl} Gallot, S.; Hulin, D.; Lafontaine, J.; Riemannian Geometry, 3rd edition, Springer,
Berlin Heidelberg (2004).

\bibitem{gordon} Gordon, M.; The Unoriented Cobordism Classes of Compact Flat Riemannian Manifolds,
J. Differential Geom. 15, no. 1, 81–90 (1980)

\bibitem{hamroy} Hamrick, G.; Royster, D; Flat Riemannian Manifolds are Boundaries, Invent. Math. 66, 405--413 (1982).


\bibitem{szz} Szczepa\'n ski, A.; Geometry of cristallographic groups, ADM, World Scientific 2012.

\bibitem{th} Thurston, W.P.; Three-Dimensional Geometry and Topology, edited by S. Levy, Princeton University Press,
Princeton (1997).

\bibitem{wolf} Wolf, J.A.; Spaces of Constant Curvature, Publish or Perish, Boston (1974).

\end{thebibliography}
\end{document}